\documentclass{birkjour}
 \newtheorem{thm}{Theorem}[section]
 \newtheorem{cor}[thm]{Corollary}
 
 \newtheorem{prop}[thm]{Proposition}
 \theoremstyle{definition}
 \newtheorem{defn}[thm]{Definition}
 \newtheorem{ex}[thm]{Example}
 \theoremstyle{remark}
 
 \numberwithin{equation}{section}
\usepackage{color}
\begin{document}

%
%
%
%
%
%
%
%
%

\title[Some properties of  dual and approximate dual of fusion frames]
 {Some properties of  dual and approximate dual of fusion frames}


\author[A. A. Arefijamaal]{Ali Akbar Arefijamaal}
\address{Department of Mathematics and Computer Sciences, Hakim Sabzevari University, Sabzevar, Iran.}
\email{arefijamaal@hsu.ac.ir; arefijamaal@gmail.com}

\author[F. Arabyani Neyshaburi]{Fahimeh Arabyani Neyshaburi}
\address{Department of Mathematics and Computer Sciences, Hakim Sabzevari University, Sabzevar, Iran.}
\email{f.arabyani@hsu.ac.ir}


\vspace{2cm}
\begin{abstract}

 In this paper we extend the notion of approximate dual to fusion frames and present some approaches to obtain dual and approximate alternate dual fusion frames. Also, we study the stability of dual and approximate alternate dual fusion frames.
\end{abstract}

\maketitle

\section{Introduction and preliminaries}

Fusion frame theory is a natural generalization of frame theory in separable Hilbert spaces, which is introduced by  Casazza and  Kutyniok in \cite{Cas04} and also  Fornasier in \cite{Forna}. Fusion frames are applied to signal processing, image processing, sampling theory, filter banks and a variety of applications that cannot be modeled by discrete frames \cite{sensor, hear}.

Let $I$ be a countable index set, recall that a sequence $\{f_{i}\}_{i\in I}$ is a \textit{frame} in a separable Hilbert space $\mathcal{H}$ if there exist constants
 $0<A\leq B<\infty$ such that
\begin{eqnarray}\label{Def frame}
A\|f\|^{2}\leq \sum_{i\in I}|\langle f,f_{i}\rangle|^{2}\leq
B\|f\|^{2},\qquad (f\in \mathcal{H}).
\end{eqnarray}
The constants $A$ and $B$ are called the \textit{lower} and \textit{upper frame bounds}, respectively.
It is said that $\{f_{i}\}_{i\in I}$ is a \textit{Bessel sequence} if the  right  inequality in  (\ref{Def frame}) is satisfied. Given a
frame $\{f_{i}\}_{i\in I}$, the \textit{frame operator} is defined by
\begin{eqnarray*}
Sf=\sum_{i\in I}\langle f,f_{i}\rangle f_i, \qquad (f\in \mathcal{H}).
\end{eqnarray*}
 It is a bounded, invertible, and self-adjoint
operator \cite{Chr08}.
The family $\{S^{-1}f_{i}\}_{i\in I}$ is also a frame for $\mathcal{H}$, the so called the \textit{canonical dual} frame. In general, a Bessel sequence $\{g_i\}_{i\in I}\subseteq\mathcal{H}$ is called an alternate \textit{dual} or simply a \textit{dual}
 for the Bessel sequence $\{f_{i}\}_{i\in I}$ if
\begin{eqnarray}\label{Def Dual}
f=\sum_{i\in I}\langle f,g_{i}\rangle f_i ,\qquad (f\in
\mathcal{H}).
\end{eqnarray}

The \textit{synthesis operator} $T: l^{2}\rightarrow \mathcal{H}$ of a Bessel sequence $\lbrace f_{i}\rbrace_{i\in I}$ is defined by $T\lbrace c_{i}\rbrace_{i\in I}= \sum_{i\in I} c_{i}f_{i}$. By (\ref{Def Dual}) two Bessel sequences $\{f_i\}_{i\in I}$ and $\{g_i\}_{i\in I}$ are dual of each other if and only if $T_GT_F^{*}=I_{\mathcal{H}}$, where $T_F$ and $T_G$ are the synthesis operators $\{f_i\}_{i\in I}$ and $\{g_i\}_{i\in I}$, respectively. For more
details on the frame theory we refer to \cite{Ca04, Chr08}.

Now we review basic definitions and primary results of fusion frames. Throughout this paper,
$\pi_V$ denotes the orthogonal projection from Hilbert space $\mathcal{H}$ onto a closed subspace $V$.
\begin{defn}
 Let $\{W_i\}_{i\in I}$ be a family of closed subspaces of $\mathcal{H}$ and $\{\omega_i\}_{i\in I}$ a family of
weights, i.e. $\omega_i>0$, $i\in I$. Then $\{(W_i,\omega_i)\}_{i\in
I}$ is called a \textit{fusion frame} for $\mathcal{H}$ if there exist the
constants $0<A\leq B<\infty$ such that
\begin{eqnarray}\label{Def. fusion}
A\|f\|^{2}\leq \sum_{i\in I}\omega_i^2\|\pi_{W_i}f\|^2\leq
B\|f\|^{2},\qquad (f\in \mathcal{H}).
\end{eqnarray}
\end{defn}
The constants $A$ and $B$ are called the \textit{fusion frame
bounds}. If we only have the upper bound in (\ref{Def. fusion}) we
call $\{(W_i,\omega_i)\}_{i\in I}$ a \textit{Bessel fusion
sequence}. A fusion frame is called $A$-\textit{tight}, if $A=B$, and \textit{Parseval} if $A= B= 1$. If
$\omega_i=\omega$ for all $i\in I$, the collection
$\{(W_i,\omega_i)\}_{i\in I}$ is called \textit{$\omega$-uniform} and we abbreviate $1$- uniform fusion frames as $\{W_i\}_{i\in I}$. A family of closed subspaces $\{W_i\}_{i\in I}$ is called an orthonormal basis for $\mathcal{H}$ when $\oplus_{i\in I} W_{i} = \mathcal{H}$ and it is a \textit{Riesz decomposition} of $\mathcal{H}$,
if for every $f\in\mathcal{H}$ there is a unique choice of $f_i\in
W_i$ such that $f =\sum_{i\in I}f_i$. Also a family of closed subspaces $\{W_i\}_{i\in I}$ is called a Riesz fusion basis whenever it is complete for $\mathcal{H}$ and there exist positive constants $A$, $B$ such that for every finite subset $J\subset I$ and arbitrary vector $f_{i}\in W_i$, we have
\begin{eqnarray*}\label{Def fusion}
A\sum_{i\in J}\| f_i\|^{2} \leq \| \sum_{i\in J} f_i\|^{2} \leq B\sum_{i\in J}\| f_i\|^{2}.
\end{eqnarray*}

It is clear that every Riesz fusion basis is a $1$- uniform fusion frame for $\mathcal{H}$, and also a fusion frame is a Riesz  basis if and only if it is a Riesz decomposition for $\mathcal{H}$, see \cite{Asgari, Cas04}.

For every fusion frame a useful local frame is proposed in the following theorem.
\begin{thm}\cite{Cas04}\label{3.5}
 Let $\{W_i\}_{i\in I}$  be a family of subspaces in $\mathcal{H}$ and $\lbrace \omega_{i}\rbrace_{i\in I}$ a family of weights. Then $\{(W_{i},\omega_{i})\}_{i\in I}$ is a fusion frame for $\mathcal{H}$ with bounds $A$ and $B$, if and only if $\{\omega_{i}\pi_{W_{i}}e_{j}\}_{i\in I,j\in J}$ is a frame for $\mathcal{H}$, with the same bounds, where $\{e_{j}\}_{j\in J}$
is an orthonormal basis for $\mathcal{H}$.
\end{thm}
Also, a connection between local and global properties is given in the next result, see \cite{Cas04}.
\begin{thm} \label{fusion 2}
For each $i\in I$, let $W_{i}$ be a closed subspace of $\mathcal{H}$ and $\omega_{i}>0$. Also let $\lbrace f_{i,j}\rbrace_{j\in J_{i}}$ be a frame for $W_{i}$ with frame bounds $\alpha_{i}$ and $\beta_{i}$ such that
\begin{eqnarray}\label{sup}
 0 < \alpha=inf_{i\in I} \alpha_{i} \leq \beta=sup_{i\in I}\beta_{i} < \infty.
\end{eqnarray}
 Then the following conditions are equivalent.
\item $(i)$
$\lbrace (W_{i},\omega_{i})\rbrace_{i\in I }$ is a fusion frame of $\mathcal{H}$ with bounds $C$ and $D$.
\item $(ii)$
$\lbrace \omega_{i}f_{i,j}\rbrace_{i\in I , j\in J_{i}}$ is a frame of $\mathcal{H}$  with bounds $\alpha C$ and $\beta D$, .
\end{thm}
Recall that for each sequence $\{W_i\}_{i\in I}$ of closed subspaces
in $\mathcal{H}$, the space
\begin{eqnarray*}
\sum_{i\in I}\oplus W_{i} =\{\{f_i\}_{i\in I}:f_i\in W_i,
\sum_{i\in I}\|f_i\|^2<\infty\},
\end{eqnarray*}
with the inner product
\begin{eqnarray*}
\langle \{f_i\}_{i\in I},\{g_i\}_{i\in I} \rangle=\sum_{i\in
I}\langle f_i,g_i \rangle,
\end{eqnarray*}
is a Hilbert space.
For a Bessel fusion sequence $\{(W_i,\omega_i)\}_{i\in I}$ of
$\mathcal{H}$, the \textit{synthesis operator} $T_{W}: \sum_{i\in
I}\oplus W_{i} \rightarrow\mathcal{H}$ is defined by
\begin{equation*}
T_{W}(\{f_i\}_{i\in I})=\sum_{i\in I}\omega_if_i,\qquad
(\{f_{i}\}_{i\in I}\in \sum_{i\in I}\oplus W_{i}).
\end{equation*}
Its adjoint operator $T_{W}^{*}: \mathcal{H}\rightarrow \sum_{i\in
I}\oplus W_{i}$, which is called the \textit{analysis
operator}, is given by
\begin{eqnarray*}
T_{W}^{*}(f)=\{\omega_{i}\pi_{W_{i}}(f)\}_{i\in I},\qquad (f\in
\mathcal{H}).\end{eqnarray*}
Let $\{(W_i,\omega_i)\}_{i\in I}$ be a fusion
frame, the \textit{fusion frame operator}
$S_{W}:\mathcal{H}\rightarrow\mathcal{H}$ is defined by $S_{W
}f=\sum_{i\in I}\omega_i^{2}\pi_{W_i}f$ is a bounded,
invertible as well as positive. Hence, we have the following
reconstruction formula \cite{Cas04}
\begin{eqnarray*}
f=\sum_{i\in I}\omega_i^2S_W^{-1}\pi_{W_i}f,\qquad (f\in \mathcal{H}).
\end{eqnarray*}
The family $\{(S_{W}^{-1}W_i,\omega_i)\}_{i\in I}$, which is also a
fusion frame, is called the \textit{canonical dual} of
$\{(W_i,\omega_i)\}_{i\in I}$. Also, a Bessel fusion sequence $\{(V_i,\nu_i)\}_{i\in I}$ is
called an \textit{alternate dual} of $\{(W_i,\omega_i)\}_{i\in I}$, \cite{Gav06} whenever
\begin{eqnarray}\label{Def:alt}
f=\sum_{i\in I}\omega_{i}\nu_{i}\pi_{V_i}S_{W}^{-1}\pi_{W_i}f,\qquad
(f\in \mathcal{H}).
\end{eqnarray}

In \cite{Gav06}, it is proved that  every alternate dual  of a fusion frame is a fusion frame. Also, easily we can see that
a Bessel fusion sequence $\lbrace (V_{i},\upsilon_{i}) \rbrace_{i\in I}$ is an alternate dual fusion frame of $\{(W_i,\omega_i)\}_{i\in I}$ if and only if $T_{V} \phi_{vw} T^{*}_{W}=I_{\mathcal{H}}$, where the bounded operator
$\phi_{vw} : \sum_{i\in I}\bigoplus W_{i} \rightarrow \sum_{i\in I}\bigoplus V_{i}$ is given by
\begin{equation}\label{phi}
\phi_{vw}(\lbrace f_{i}\rbrace_{i\in I}) = \lbrace \pi_{V_{i}}S^{-1}_{W}f_{i}\rbrace_{i\in I}.
\end{equation}
 Moreover, a Bessel fusion sequence $V=\{(V_{i},\omega_{i})\}_{i\in I}$ given by $V_{i} = S_{W}^{-1}W_{i} \oplus U_{i}$, is  an alternate dual fusion frame of $\lbrace (W_{i}, \omega_{i})\rbrace_{i\in I}$ in which $U_{i}$ is a closed subspace of $\mathcal{H}$ for all $i\in I$, \cite{Leng}. Recently, Heineken et al.  introduced the other concept of dual fusion
frames,
\cite{Hei14}. For two fusion frames $\lbrace (W_{i}, \omega_{i})\rbrace_{i\in I}$ and $\lbrace (V_{i}, \upsilon_{i})\rbrace_{i\in I}$ if there exists a mapping  $Q\in B(\sum_{i\in I}\bigoplus W_{i},\sum_{i\in I}\bigoplus V_{i})$, such that $T_{V} Q T^{*}_{W} = I_{\mathcal{H}}$ then $\lbrace (V_{i}, \upsilon_{i})\rbrace_{i\in I}$ is called a $Q$-dual of $\lbrace (W_{i}, \omega_{i})\rbrace_{i\in I}$. Clearly every alternate  dual fusion frame is a $\phi_{vw}$-dual. $Q$-duals are useful tools for establishing reconstruction formula. For more information on fusion frames we refer the reader to \cite{Asgari, Cas04, Cas08}.

\section{Approximate duals}
Alternate dual fusion frames play a key role in fusion frame theory, however, their explicit computations seem rather intricate. In this section, we introduce the notion of approximate dual for fusion frames and discuss the existence of alternate dual fusion frames from an approximate alternate dual. Moreover, we present a complete characterization of alternate duals of Riesz fusion bases. The notion of approximate  dual for discrete frames has been already introduced by Christensen and Laugesen in \cite{app.} and then for $g$-frames in \cite{khosravi}, however many of theirs results are invalid for fusion frames. Throughout this section we consider a Riesz fusion basis as a $1$-uniform fusion frame.

First, we recall the notion of approximate dual for discrete frames. Let $F = \lbrace f_{i}\rbrace_{i\in I}$ and $G = \lbrace g_{i}\rbrace_{i\in I}$ be Bessel sequences for $\mathcal{H}$. Then $F$ and $G$ are called \textit{approximate dual frames} if $\Vert I_{\mathcal{H}} - T_GT_F^{*}\Vert < 1$. In this case $\lbrace (T_GT_F^{*})^{-1}g_i\rbrace$ is a dual of $F$, see \cite{app.}.

Now we introduce approximate duality for fusion frames.
\begin{defn}
Let $\lbrace (W_{i}, \omega_{i})\rbrace_{i\in I}$ be a  fusion frame. A Bessel fusion sequence $\lbrace (V_{i}, \upsilon_{i})\rbrace_{i\in I}$ is called an approximate alternate dual of $\lbrace (W_{i}, \omega_{i})\rbrace_{i\in I}$ if
 \begin{eqnarray*}
\Vert I_{\mathcal{H}} - T_{V} \phi_{vw} T^{*}_{W}\Vert < 1.
\end{eqnarray*}
\end{defn}
Putting
\begin{eqnarray}
 \psi_{vw} = T_{V} \phi_{vw} T^{*}_{W}.
 \end{eqnarray}
Then, we have the following reconstruction formula
 \begin{eqnarray*}
f = \sum_{i\in I} (\psi_{vw})^{-1} \omega_{i}\upsilon_{i}\pi_{V_{i}} S_{W}^{-1} \pi_{W_{i}} f = \sum_{n=0}^{\infty} (I-\psi_{vw})^{n} \psi_{vw}f, \qquad (f\in \mathcal{H}).
\end{eqnarray*}
\begin{prop}
Let $V = \lbrace (V_{i}, \upsilon_{i})\rbrace_{i\in I}$ be an approximate alternate dual of a fusion frame $W = \lbrace (W_{i}, \omega_{i})\rbrace_{i\in I}$. Then $V$ is a fusion frame.
\end{prop}
\begin{proof}
Let $B$ and $D$ be Bessel bounds of $W$ and $V$, respectively. Then
\begin{eqnarray*}
\Vert \psi_{vw}^{*}f\Vert^{2} &=& \sup_{\Vert g\Vert=1} \vert \langle T_{W}\phi_{vw}^{*}T^{*}_{V} f, g\rangle\vert^{2}\\
&=& \sup_{\Vert g\Vert=1} \vert \langle \sum_{i\in I} \omega_{i}\upsilon_{i} \pi_{W_{i}} S_{W}^{-1} \pi_{V_{i}}f, g\rangle\vert^{2}\\
&\leq& \sup_{\Vert g\Vert=1}\sum_{i\in I} \upsilon_{i}^{2} \Vert \pi_{V_{i}}f\Vert^{2} \sum_{i\in I} \omega_{i}^{2} \Vert S_{W}^{-1} \pi_{W_{i}}g\Vert^{2} \\
&\leq& \Vert S_{W}^{-1} \Vert^{2} B \sum_{i\in I} \upsilon_{i}^{2} \Vert \pi_{V_{i}}f\Vert^{2},
\end{eqnarray*}
for every $f\in \mathcal{H}$.
This follows that
 \begin{eqnarray*}
\Vert f\Vert^{2} \dfrac{\Vert (\psi_{vw}^{-1})^{*}\Vert^{-2}}{\Vert S_{W}^{-1}\Vert^{2}B} \leq \sum_{i\in I} \upsilon_{i}^{2} \Vert \pi_{V_{i}}f\Vert^{2} \leq D \Vert f\Vert^{2}.
 \end{eqnarray*}
\end{proof}

The following proposition describes approximate duality of fusion frames with respect to local frames.
\begin{prop}\label{app.wrt.local}
Let $\lbrace e_{j}\rbrace_{j\in J}$ be an orthonormal basis of $\mathcal{H}$. Then the Bessel sequence $V = \lbrace (V_{i}, \upsilon_{i})\rbrace_{i\in I} $ is an approximate alternate dual fusion frame of a fusion frame  $W = \lbrace (W_{i}, \omega_{i})\rbrace_{i\in I}$ if and only if $\lbrace \upsilon_{i} \pi_{V_{i}}e_{j}\rbrace_{i\in I, j\in J}$ is an approximate dual of $\lbrace \omega_{i} \pi_{W_{i}} S_{W}^{-1}e_{j}\rbrace_{i\in I, j\in J}$.
\end{prop}
\begin{proof}
For each $f\in\mathcal{H}$ we have
\begin{eqnarray*}
\sum_{i\in I, j\in J} \vert \langle f, \omega_{i}\pi_{W_{i}}S_{W}^{-1}e_j\rangle\vert^{2} &=& \sum_{i\in I}\sum_{j\in J} \vert \langle \omega_{i}S_{W}^{-1}\pi_{W_i}f, e_j\rangle\vert^{2}\\
&=& \sum_{i\in I} \omega_{i}^{2} \|S_{W}^{-1}\pi_{W_i}f\|^2 \\
&\leq& \|S_{W}^{-1}\|^2 \sum_{i\in I} \omega_{i}^{2} \|\pi_{W_i}f\|^2.
\end{eqnarray*}
This implies that $F = \{\omega_{i}\pi_{W_{i}}S_{W}^{-1}e_{j}\}_{i\in I, j\in J}$ is a Bessel sequence for $\mathcal{H}$. Similarly, $G = \{\upsilon_{i} \pi_{V_{i}}e_{j}\}_{i\in I, j\in J}$ is also Bessel sequence for $\mathcal{H}$. Moreover,
\begin{eqnarray*}
 T_{G} T_{F}^{*}f&=& \sum_{i\in I, j\in J}\langle f,\omega_{i}\pi_{W_{i}}S_{W}^{-1}e_{j}\rangle \upsilon_{i} \pi_{V_{i}}e_{j}\\
&=&\sum_{i\in I, j\in J}\omega_{i}\upsilon_{i}\pi_{V_{i}}\langle S_{W}^{-1}\pi_{W_{i}}f, e_{j}\rangle e_{j}\\
&=&\sum_{i\in I}\omega_{i}\upsilon_{i}\pi_{V_{i}}S_{W}^{-1}\pi_{W_{i}}f\\
&=& T_{V} \phi_{vw} T_{W}^{*}f =\psi_{vw}.
\end{eqnarray*}
for all $f\in\mathcal{H}$. This completes the proof.
\end{proof}

\begin{thm}\label{3.8}
Let $W = \lbrace (W_{i}, \omega_{i})\rbrace_{i\in I} $ be a fusion frame  and $V = \lbrace (V_{i}, \upsilon_{i})\rbrace_{i\in I}$ be a Bessel fusion sequence, also $\lbrace g_{i,j}\rbrace_{j\in J_{i}}$ be a frame for $V_{i}$ with bounds $A_{i}$ and $B_{i}$ for every $i\in I$ such that $0 < a = \inf_{i\in I} A_{i} \leq \sup_{i\in I}B_{i} = b< \infty$. Then $V$ is an approximate alternate dual fusion frame of $W$ if and only if $G = \lbrace \upsilon_{i}g_{i,j}\rbrace_{i\in I, j\in J_{i}}$ is an approximate dual of $F = \lbrace \omega_{i}\pi_{W_{i}}S_{W}^{-1}\widetilde{g}_{i,j}\rbrace_{i\in I, j\in J_{i}}$ where $\lbrace \widetilde{g}_{i,j}\rbrace_{j\in J_{i}}$ is the canonical dual of $\lbrace g_{i,j}\rbrace_{j\in J_{i}}$.
\end{thm}
\begin{proof} We first show that $F$  is a  Bessel sequence for $\mathcal{H}$.
Indeed, for each $f\in \mathcal{H}$ \begin{eqnarray*} \sum_{i\in I, j\in
J_{i}} \vert \langle f,
\omega_{i}\pi_{W_{i}}S_{W}^{-1}\widetilde{g}_{i,j}\rangle\vert^{2} &=&
\sum_{i\in I}\omega_{i}^{2}\sum_{j\in J_{i}} \vert \langle
\pi_{V_{i}}S_{W}^{-1}\pi_{W_{i}}f, \widetilde{g}_{i,j}\rangle\vert^{2}\\
&\leq& \sum_{i\in I} \dfrac{\omega_{i}^{2}}{A_{i}} \Vert
\pi_{V_{i}}S_{W}^{-1}\pi_{W_{i}}f\Vert^{2}\\
&\leq& \dfrac{\Vert S_{W}^{-1}\Vert^{2}}{a} \sum_{i\in I} \omega_{i}^{2}
\Vert \pi_{W_{i}}f\Vert^{2}. \end{eqnarray*}
Moreover, by Theorem \ref{fusion 2}, $G$ is a Bessel sequence for $\mathcal{H}$.
 On the other hand,
\begin{eqnarray*} T_{V}\phi_{vw}T^{*}_{W}f &=& \sum_{i\in I}
\omega_{i}\upsilon_{i} \pi_{V_{i}}S_{W}^{-1}\pi_{W_{i}}f\\ &=& \sum_{i\in I}
\omega_{i}\upsilon_{i} \sum_{j\in J_{i}} \langle
\pi_{V_{i}}S_{W}^{-1}\pi_{W_{i}}f,
\widetilde{g}_{i,j}\rangle g_{i,j}\\
&=& \sum_{i\in I, j\in J_{i}} \langle f,
\omega_{i}\pi_{W_{i}}S_{W}^{-1}\widetilde{g}_{i,j}\rangle \upsilon_{i}g_{i,j}
= T_{G}T^{*}_{F}f. \end{eqnarray*}
This completes the proof.\end{proof}

The following theorem, gives the idea
that will lead to one of the main results of this section.

\begin{thm}\label{Riesz app} Let $W=\lbrace (W_{i}, \omega_{i})\rbrace_{i\in I}$ be a Riesz
fusion basis. For an approximate alternate dual fusion frame $\lbrace (V_{i}, \omega_{i})\rbrace_{i\in I}$ of $W$, the sequence $\lbrace (\psi_{vw}^{-1} V_{i},
\omega_{i})\rbrace_{i\in I}$ is an alternate dual fusion frame of $W$.
\end{thm} \begin{proof}
Suppose that $\lbrace e_{i,j}\rbrace_{j\in J_{i}}$ is
an orthonormal basis of $W_{i}$, for each $i\in I$. Then $F := \lbrace
\omega_{i}e_{i,j}\rbrace_{i\in I,j\in J_{i}}$ is a frame for $\mathcal{H}$ by
Theorem \ref{fusion 2}. Now, for each $f\in
\mathcal{H}$ we obtain \begin{eqnarray*} \sum_{i\in I, j\in J_{i}} \vert
\langle f,\omega_{i}\pi_{V_{i}}S_{W}^{-1}e_{i,j}\rangle\vert^{2} &=&
\sum_{i\in I}\sum_{j\in J_i} \vert \langle \omega_{i}S_{W}^{-1}\pi_{V_i}f,
e_{i,j}\rangle\vert^{2}\\ &\leq& \sum_{i\in I} \omega_{i}^{2}\|S_{W}^{-1}\pi_{V_i}f\|^2 \\
&\leq& \|S_{W}^{-1}\|^2 \sum_{i\in I} \omega_{i}^{2}\|\pi_{V_i}f\|^2.
\end{eqnarray*} Thus, $G := \lbrace \omega_{i}\pi_{V_{i}} S_{W}^{-1}
e_{i,j}\rbrace_{i\in I,j\in J_{i}}$ is a Bessel sequence for
$\mathcal{H}$. Moreover,

\begin{eqnarray*}
\psi_{vw}f = \sum_{i\in I} \omega_{i}^{2}\pi_{V_{i}} S_{W}^{-1}
\pi_{W_{i}}f &=& \sum_{i\in I,j\in J_{i}} \omega_{i}^{2}\pi_{V_{i}}
S_{W}^{-1} \langle f, e_{i,j}\rangle e_{i,j}\\ &=& \sum_{i\in I,j\in
J_{i}} \langle f, \omega_{i}e_{i,j}\rangle \omega_{i}\pi_{V_{i}}
S_{W}^{-1} e_{i,j}=T_GT_F^*f. \end{eqnarray*} Hence,  by assumption $G$ is an approximate  dual of $F$. This
implies that the sequence $\lbrace (T_{G}T_{F}^{*})^{-1} \omega_{i}\pi_{V_{i}} S_{W}^{-1}
e_{i,j}\rbrace_{i\in I,j\in J_{i}}$ is a dual for $\lbrace
\omega_{i}e_{i,j}\rbrace_{i\in I,j\in J_{i}}$. On the other hand the sequence
$\lbrace \omega_{i}e_{i,j}\rbrace_{i\in I, j\in J_{i}}$ is a Riesz basis for
$\mathcal{H}$  by Theorem 3.6 of \cite{Asgari}. Using the fact that discrete
Riesz bases have only one dual we obtain
\begin{eqnarray}\label{psi.TGTF*}
  (T_{G}T_{F}^{*})^{-1} \omega_{i}\pi_{V_{i}} S_{W}^{-1} e_{i,j} = S_{F}^{-1} \omega_{i}e_{i,j} \quad (i\in I, j\in J_{i}).
\end{eqnarray} Furthermore, it is not difficult to see that $S_{F} = S_{W}$. Indeed, for all $f\in \mathcal{H}$ we have
\begin{eqnarray*}
 S_{F}f &=& \sum_{i\in I, j\in J_{i}} \langle f, \omega_{i}e_{i,j}\rangle \omega_{i}e_{i,j}\\
&=& \sum_{i\in I} \omega_{i}^{2}\sum_{j\in J_{i}}\langle \pi_{W_{i}}f, e_{i,j}\rangle e_{i,j}\\
&=& \sum_{i\in I} \omega_{i}^{2}\pi_{W_{i}}f = S_{W}f.
\end{eqnarray*}
Now, since $T_{G}T^{*}_{F}=\psi_{vw}$. By
Substituting $\psi_{vw}$ and $S_{W}$ in (\ref{psi.TGTF*}), we finally
conclude that \begin{eqnarray*} \psi_{vw}^{-1} \pi_{V_{i}} S_{W}^{-1} e_{i,j}
= S_{W}^{-1}e_{i,j},\quad (i\in I, j\in J_{i}). \end{eqnarray*} In
particular, \begin{eqnarray}\label{Riesz..} \psi_{vw}^{-1} V_{i} \supseteq
S_{W}^{-1} W_{i}, \quad (i\in I). \end{eqnarray} It immediately follows that
$\lbrace (\psi_{vw}^{-1} V_{i}, \omega_{i})\rbrace_{i\in I}$ is an alternate dual fusion frame of $\lbrace (
W_{i}, \omega_{i})\rbrace_{i\in I}$. \end{proof} By the above theorem we obtain the following
 characterization of alternate duals of  Riesz fusion bases.
\begin{cor}\label{al.Riesz} Let $W=\lbrace (W_{i},\omega_{i})\rbrace_{i\in I}$ be a  Riesz
fusion basis. A Bessel sequence $V = \lbrace (V_{i},\omega_{i})\rbrace_{i\in I}$
is an  alternate dual fusion frame of $W$ if and only if \begin{eqnarray}\label{Riesz dual}
 V_{i} \supseteq S_{W}^{-1} W_{i}, \quad (i\in I).
\end{eqnarray} \end{cor}
\begin{proof}
If $V$ satisfies in (\ref{Riesz dual}), clearly $V$ is an alternate dual of $W$. On the other hand, since every dual fusion frame is an approximate dual with $\psi_{vw} = I_{\mathcal{H}}$, so  by (\ref{Riesz..}) the result follows.
\end{proof}
 Corollary \ref{al.Riesz}, also  shows that, unlike discrete frames, Riesz fusion bases may have more than one dual. Moreover, in the next proposition, we show that every fusion frame has at least an alternate dual.
\begin{prop}
Every fusion frame has an alternate dual fusion frame different from the canonical dual fusion frame.
\end{prop}
\begin{proof}
Let  $\{(W_{i},\omega_{i})\}_{i\in I}$ be a fusion frame with frame operator $S_{W}$. First suppose that there exists $i_{0}\in I$ such that $W_{i_{0}} \neq \mathcal{H}$. Take $V_{i}=S_{W}^{-1}W_{i}$ for $i\neq i_{0}$ and $V_{i_{0}}=S_{W}^{-1}W_{i_{0}} \oplus U_{i_{0}}$ where $U_{i_{0}}\subseteq (S_{W}^{-1}W_{i_{0}})^{\bot}$ is an arbitrary closed subspace.
Obviously $\{(V_{i},\omega_{i})\}_{i\in I}$ is an alternate dual fusion frame of $\{(W_{i},\omega_{i})\}_{i\in I}$. Indeed
\begin{eqnarray*}
T_{V} \phi_{vw} T^{*}_{W}f &=& \sum_{i\in I, i\neq i_{0}}\omega_{i}^{2}\pi_{V_{i}}S_{W}^{-1}\pi_{W_{i}}f + \omega_{i_{0}}^{2}\pi_{V_{i_{0}}}S_{W}^{-1}\pi_{W_{i_{0}}}f\\
&=& \sum_{i\in I, i\neq i_{0}}\omega_{i}^{2}\pi_{S_{W}^{-1}W_{i}}S_{W}^{-1}\pi_{W_{i}}f + \omega_{i_{0}}^{2}\pi_{S_{W}^{-1}W_{i_{0}} \oplus U_{i_{0}}}S_{W}^{-1}\pi_{W_{i_{0}}}f\\
&=& \sum_{i\in I} \omega_{i}^{2} \pi_{S_{W}^{-1}W_{i}}S_{W}^{-1}\pi_{W_{i}}f = f,
\end{eqnarray*}
for every $f\in \mathcal{H}$. On the other hand, assume that  $W_{i}=\mathcal{H}$ for all $i\in I$. It immediately follows that $\lbrace \omega_{i}\rbrace_{i\in I}\in l^{2}$. Take
$V_{1}=\mathcal{H}$ and $V_{i}=\{0\}$ for $i > 1$, and assume that $\nu_{1}=\frac{\sum_{i\in I}\omega_{i}^{2}}{\omega_{1}}$ and $ \nu_{i}=\omega_{i}$ for $ i > 1$. Then
 $S_{W}f=\left(\sum_{i\in I}\omega_{i}^{2}\right)f$ and for every $f\in \mathcal{H}$ we have
\begin{eqnarray*}
\sum_{i\in I}\omega_{i}\nu_{i}\pi_{V_{i}}S_{W}^{-1}\pi_{W_{i}}f&=&\sum_{i\in I}\omega_{i}\nu_{i}\pi_{V_{i}}S_{W}^{-1}f = \omega_{1}\nu_{1}\pi_{V_{1}}S_{W}^{-1}f\\
&=& \left(\sum_{i\in I}\omega_{i}^{2}\right) S_{W}^{-1}f = f.
\end{eqnarray*}
This shows that $\{(V_{i},\nu_{i})\}_{i\in I}$ is an alternate dual fusion frame of $\{(W_{i},\omega_{i})\}_{i\in I}$.
\end{proof}
\begin{ex}
Let $W = \lbrace W_{i}\rbrace_{i\in I}$ be a Riesz  fusion basis. Consider
\begin{eqnarray*}
V_{i} = (\overline{\textit{span}}_{j\neq i}\lbrace W_{j}\rbrace)^{\perp}, \quad (i\in I)
\end{eqnarray*}
we claim that $V = \lbrace (V_{i}, \upsilon_{i})\rbrace_{i\in I}$ is an alternate dual of $\lbrace W_{i}\rbrace_{i\in I}$, for all $\lbrace \upsilon_{i}\rbrace_{i\in I} \in l^{2}$. Take $f_{i}\in W_{i}$, since $\lbrace S_{W}^{-1/2}W_{i}\rbrace_{i\in I}$ is an orthogonal family of subspaces in $\mathcal{H}$ so $S_{W}^{-1}f_{i}\in V_{i}$. Hence $ V_{i} \supseteq S_{W}^{-1} W_{i}$, for every $i\in I$ and so $V$ is a dual of $W$ by Corollary \ref{al.Riesz}. In fact, this dual is the unique maximal biorthogonal sequence for $\lbrace W_{i}\rbrace_{i\in I}$, see also Proposition 4.3 in \cite{Cas04}.
\end{ex}
Suppose that $\lbrace W_{i}\rbrace_{i\in I}$ is a Riesz fusion basis, by Theorem 3.9 in \cite{Asgari}, there exists an orthonormal fusion basis $\lbrace U_{i}\rbrace_{i\in I}$ and a bounded bijective linear operator $T: \mathcal{H}\rightarrow \mathcal{H}$ for which $T U_{i} = W_{i}$ for all $i\in I$. Therefore the canonical dual of a Riesz fusion basis is also a Riesz fusion basis. The following theorem,   shows that other alternate duals of $\lbrace W_{i}\rbrace_{i\in I}$ are not Riesz fusion basis. This result is a generalization of Proposition 3.7 (2) in \cite{Hei15} for infinite case.

\begin{thm}
Let $W = \lbrace W_{i}\rbrace_{i\in I}$ be a Riesz fusion basis. The only dual $\lbrace V_{i}\rbrace_{i\in I}$ of $W$ that is Riesz  basis is the canonical dual.
\end{thm}
\begin{proof}
Suppose the Riesz basis $\lbrace V_{i}\rbrace_{i\in I}$ is an  alternate dual fusion frame of $W$. By Corollary \ref{al.Riesz}, $S_{W}^{-1}W_{i} \subseteq V_{i}$ for all $i\in I$. Assume that there exists $j\in I$ such that $S_{W}^{-1}W_{j} \subset V_{j}$, pick up a non-zero $0 \neq f \in V_{j}\cap (S_{W}^{-1}W_{j})^{\perp}$. Since $\lbrace S_{W}^{-1}W_{i}\rbrace_{i\in I}$ is a Riesz fusion basis we can choose a unique sequence $\lbrace g_{i}\rbrace_{i\in I}$ such that $f = \sum_{i\in I} g_{i}$ where $g_{i}\in S_{W}^{-1}W_{i}$ for all $i\in I$. Therefore, the vector $f$ has two representations of the elements in the Riesz fusion basis $\lbrace V_{i}\rbrace_{i\in I}$ which is a contradiction. Hence $V_{i} = S_{W}^{-1} W_{i}$ for every $i\in I$.
\end{proof}

Suppose $L\in B(\mathcal{H})$ is invertible and $\lbrace (V_{i}, \upsilon_{i})\rbrace_{i\in I}$ is an alternate dual (approximate alternate dual)  fusion frame of $W = \lbrace (W_{i}, \omega_{i})\rbrace_{i\in I}$. It is natural to ask whether $\lbrace (L V_{i}, \upsilon_{i})\rbrace_{i\in I}$ is an alternate dual (approximate alternate dual)  fusion frame  of $\lbrace (L W_{i}, \omega_{i})\rbrace_{i\in I}$.
\begin{thm}
Let $W = \lbrace (W_{i}, \omega_{i})\rbrace_{i\in I}$ be a fusion frame and $L\in B(\mathcal{H})$ be invertible such that $L^{*}L W_{i} \subseteq W_{i}$ for every $i\in I$. The following statements  hold:\\
\item[(i)]
If $V = \lbrace (V_{i}, \upsilon_{i})\rbrace_{i\in I}$ is an alternate dual fusion frame of $W$. Then the sequence $LV = \lbrace (L V_{i}, \upsilon_{i})\rbrace_{i\in I}$ is an alternate dual fusion frame of $LW = \lbrace (L W_{i}, \omega_{i})\rbrace_{i\in I}$.\\
\item[(ii)]
If $V$ is an approximate alternate dual fusion frame of $W$ such that,
\begin{eqnarray*}
\Vert I_{\mathcal{H}} - \psi_{vw}\Vert < \Vert L\Vert^{-1} \Vert L^{-1}\Vert^{-1}.
\end{eqnarray*}
Then $LV$ is an approximate alternate dual fusion frame of $LW$.
\end{thm}
\begin{proof} The sequence $\lbrace (L W_{i}, \omega_{i})\rbrace_{i\in I}$ is a fusion frame with the frame operator $LS_{W}L^{-1}$ and $\pi_{LW_{i}} = L \pi_{W_{i}} L^{-1}$, see \cite{Cas08}. Therefore, for each $f\in\mathcal{H}$ we obtain
\begin{eqnarray*}
\sum_{i\in I} \omega_{i}\upsilon_{i}\pi_{LV_{i}} S_{LW} ^{-1}\pi_{LW_{i}}f &=& \sum_{i\in I}\omega_{i}\upsilon_{i} L \pi_{V_{i}} S_{W}^{-1}\pi_{W_{i}} L^{-1}f\\
&=& L L^{-1}f = f.
\end{eqnarray*}
This proves (i). To show (ii) first note that $\psi_{Lv,Lw} = L \psi_{vw} L^{-1}$, hence
\begin{eqnarray*}
\Vert (I_{\mathcal{H}} - \psi_{Lv,Lw})f\Vert &=& \Vert (I_{\mathcal{H}} - L \psi_{v,w} L^{-1})f\Vert\\
&=& \Vert L(I_{\mathcal{H}} - \psi_{v,w} )L^{-1}f\Vert\\
& <& \Vert f\Vert
\end{eqnarray*}
for all $f\in\mathcal{H}$. This follows the result.
\end{proof}
\section{Stability of approximate duals}
In frame theory, every $f\in \mathcal{H}$ is represented by the collection of
coefficients $\{ \langle f,f_{i}\rangle\}_{i\in I}$. From these coefficients,
$f$ can be recovered using a reconstruction formula by dual frames.  In real
applications, in these transmissions usually a part of the data vectors
change or reshape in the other words, disturbances affect on the information.
In this respect,  stability of frames and dual frames under perturbations have a key role in
practice.
The stability of approximate dual of discrete frames and g-frames can be found in \cite{app., khosravi}. In the following, we discuss on the stability of approximate alternate dual fusion frames  under some perturbations. First, we fix  definition of perturbation.
\begin{defn}
Let $\lbrace W_{i}\rbrace_{i\in I}$ and $\lbrace \widetilde{W}_{i}\rbrace_{i\in I}$ be closed subspaces in $\mathcal{H}$. Also let $\{ \omega_{i}\}_{i\in I}$ be positive numbers and $\epsilon > 0$. We call $\lbrace (\widetilde{W}_{i}, \omega_{i})\rbrace_{i\in I}$ a $\epsilon$-perturbation of $ \lbrace (W_{i}, \omega_{i})\rbrace_{i\in I}$ whenever for every $f\in \mathcal{H}$,
 \begin{eqnarray*}
\sum_{i\in I} \omega_{i}^{2} \Vert (\pi_{\widetilde{W}_{i}} - \pi_{W_{i}}) f\Vert^{2} < \epsilon \Vert f\Vert^{2}.
\end{eqnarray*}
\end{defn}
\begin{thm}\label{perturbation}
Let $V = \lbrace (V_{i}, \upsilon_{i})\rbrace_{i\in I}$ be  an approximate alternate dual fusion frame of a fusion frame $W = \lbrace (W_{i},\omega_{i})\rbrace_{i\in I}$. Also let $\lbrace (U_{i},\upsilon_{i})\rbrace_{i\in I}$ be a $\epsilon$-perturbation of $V$  , such that
 \begin{eqnarray}\label{M}
\epsilon < \left(\dfrac{ 1 - \Vert I_{\mathcal{H}} - \psi_{vw}\Vert }{\sqrt{B} \Vert S_{W}^{-1}\Vert}\right)^{2},
\end{eqnarray}
where $B$ is the upper bound of $W$. Then $\lbrace (U_{i}, \upsilon_{i})\rbrace_{i\in I}$ is also an approximate alternate dual fusion frame of $W$.
In particular, if $W$ is a Parseval fusion frame and choose $V=W$, then the result holds for $\epsilon<1$.
\end{thm}
\begin{proof}
Notice that $\lbrace (U_{i}, \upsilon_{i})\rbrace_{i\in I}$ is a Bessel fusion sequence, in fact
\begin{eqnarray*}
\sum_{i\in I} \upsilon_{i}^{2} \Vert \pi_{U_{i}} f\Vert^{2} &=& \sum_{i\in I} \upsilon_{i}^{2} \Vert \pi_{V_{i}} f + (\pi_{U_{i}} - \pi_{V_{i}})f\Vert^{2} \\
 &\leq& \left(\left(\sum_{i\in I} \upsilon_{i}^{2} \Vert \pi_{V_{i}}f\Vert^{2}\right)^{1/2} + \left(\sum_{i\in I} \upsilon_{i}^{2} \Vert (\pi_{U_{i}} - \pi_{V_{i}}) f\Vert^{2}\right)^{1/2}\right)^{2} \\
&\leq& (\sqrt{D}+\sqrt{\epsilon})^{2} \Vert f\Vert^{2},
\end{eqnarray*}
where $D$ is the upper bound of $V$. On the other hand,
\begin{eqnarray*}
\Vert (I_{\mathcal{H}} -\psi_{uw})f\Vert &\leq& \Vert (I_{\mathcal{H}} - \psi_{vw})f\Vert + \Vert (\psi_{vw} - \psi_{uw})f\Vert\\
&\leq& \Vert (I_{\mathcal{H}} - \psi_{vw})f\Vert + \sup_{\Vert g\Vert=1}\left\vert {\left\langle \sum_{i\in I} \omega_{i} \upsilon_{i}(\pi_{V_{i}} - \pi_{U_{i}})S_{W}^{-1}\pi_{W_{i}}f, g\right\rangle}\right\vert \\
&=& \Vert (I_{\mathcal{H}} - \psi_{vw})f\Vert + \sup_{\Vert g\Vert=1} \left\vert{\sum_{i\in I} \left\langle \omega_{i} S_{W}^{-1}\pi_{W_{i}}f, \upsilon_{i} (\pi_{V_{i}} - \pi_{U_{i}})g\right\rangle}\right\vert \\
&\leq& \Vert (I_{\mathcal{H}} - \psi_{vw})f\Vert + \sqrt{\epsilon} \left(\sum_{i\in I} \omega_{i}^{2} \Vert S_{W}^{-1}\pi_{W_{i}}f\Vert^{2}\right)^{1/2} \\
&\leq& \Vert (I_{\mathcal{H}} - \psi_{vw})f\Vert +  \sqrt{\epsilon B} \Vert S_{W}^{-1}\Vert \Vert f\Vert< \Vert f\Vert,
\end{eqnarray*}
where the last inequality is implied from (\ref{M}). The rest follows by the fact that each Parseval fusion frame is a dual of itself.
\end{proof}
\begin{ex} Consider
\begin{eqnarray*}
W_{1} = \mathbb{R}^{2} \times \{0\}, \quad W_{2} = \{0\}\times\mathbb{R}^{2}, \quad W_{3} = {\textit{span}}\{(1,0,0)\},
\end{eqnarray*}
\begin{eqnarray*}
V_{1} =  {\textit{span}}\{(0,1,0)\}, \quad V_{2} = \{0\} \times \mathbb{R}^{2}, \quad V_{3} = {\textit{span}}\{(1,0,0)\}.
\end{eqnarray*}
Then $W = \lbrace W_{i}\rbrace_{i=1}^{3}$ is a fusion frame and $\| S_{W}^{-1}\| = 1$ Also, we have $\| I_{\mathcal{H}} - \psi_{vw}\| = \frac{1}{2}$ and so, the Bessel sequence $V = \lbrace V_{i}\rbrace_{i=1}^{3}$ is an approximate alternate dual fusion frame of $W$. Now, if we take
\begin{eqnarray*}
U_{1} =  V_{1}, \quad U_{2} = V_{2}, \quad U_{3} = {\textit{span}}\{(\alpha,\beta,0)\},
\end{eqnarray*}
where $\frac{1}{2} \leq \alpha < 1$ and $0 \leq \beta \leq \frac{1}{100}$, then $U = \lbrace U_{i}\rbrace_{i\in I}$ is a $\epsilon$-perturbation of $V$ with $\epsilon < \frac{1}{8}$. Hence, by Theorem \ref{perturbation}, $U$ is also an approximate alternate dual fusion frame of $W$.
\end{ex}

The next result is obtained immediately from Theorem \ref{perturbation}.
\begin{cor}
Let $\lbrace (V_{i}, \upsilon_{i})\rbrace_{i\in I}$ be  an alternate dual fusion frame of a  fusion frame $W = \lbrace (W_{i}, \omega_{i})\rbrace_{i\in I}$. Also let $\lbrace (U_{i},\upsilon_{i}))\rbrace_{i\in I}$ be a $\epsilon$-perturbation of $V$, and
 \begin{eqnarray}\label{*}
 \sqrt{\epsilon B} \leq \dfrac{1}{\Vert S_{W}^{-1}\Vert},
\end{eqnarray}
where $B$ is the upper bound of $W$. Then $\lbrace (U_{i}, \upsilon_{i})\rbrace_{i\in I}$ is an approximate alternate dual fusion frame of $W$.
\end{cor}

\begin{thm}\label{perturbation2}
Let $V = \lbrace (V_{i}, \upsilon_{i})\rbrace_{i\in I}$ be  an approximate alternate dual fusion frame of a fusion frame $W = \lbrace (W_{i}, \omega_{i})\rbrace_{i\in I}$. Also let $ \lbrace U_{i}\rbrace_{i\in I}$ be a $\epsilon$-perturbation of $W$ with
\begin{eqnarray}
\sqrt{\epsilon} < \dfrac{1 - (\sqrt{BD}\Vert S_{W}^{-1} - S_{U}^{-1}\Vert + \Vert I_{\mathcal{H}} - \psi_{vw}\Vert)}{\sqrt{D} \Vert S_{U}^{-1}\Vert},
\end{eqnarray}
where $B$ and $D$ are the upper bounds of $W$ and $V$ , respectively.
 Then $\lbrace (V_{i}, \upsilon_{i})\rbrace_{i\in I}$ is also an approximate alternate dual fusion frame of $U = \lbrace (U_{i}, \omega_{i})\rbrace_{i\in I}$.
\end{thm}
\begin{proof}
Applying the Cauchy Schwarz inequality for every $f\in \mathcal{H}$ we have
\begin{eqnarray*}
\Vert (I_{\mathcal{H}} - \psi_{vu})f\Vert &\leq& \Vert (I_{\mathcal{H}} - \psi_{vw})f\Vert + \Vert (\psi_{vw} - \psi_{vu})f\Vert\\
&\leq& \Vert (I_{\mathcal{H}} - \psi_{vw})f\Vert + \sup_{\Vert g\Vert = 1}\left\vert \left\langle \sum_{i\in I}\omega_{i}\upsilon_{i} \pi_{V_{i}}(S_{W}^{-1}\pi_{W_{i}} - S_{U}^{-1}\pi_{U_{i}})f, g \right\rangle \right\vert\\
&\leq& \Vert (I_{\mathcal{H}} - \psi_{vw})f\Vert + \sup_{\Vert g\Vert = 1}\left\vert \left\langle \sum_{i\in I}\omega_{i}\upsilon_{i}(S_{W}^{-1} - S_{U}^{-1})\pi_{W_{i}}f, \pi_{V_{i}}g \right\rangle\right\vert\\
&+& \sup_{\Vert g\Vert = 1} \left\vert \left\langle \sum_{i\in I}\omega_{i}\upsilon_{i} S_{U}^{-1}(\pi_{W_{i}} - \pi_{U_{i}})f, \pi_{V_{i}}g \right\rangle \right\vert\\
&\leq& \Vert (I_{\mathcal{H}} - \psi_{vw})f\Vert + \sqrt{D}\left(\Vert S_{W}^{-1} - S_{U}^{-1}\Vert \sqrt{B}+\sqrt{\epsilon}\Vert S_{U}^{-1}\Vert\right)\Vert f\Vert\\
&<& \Vert f\Vert,
\end{eqnarray*}
where the last inequality is obtained by the assumption.

\end{proof}

\begin{ex}\label{example}
Consider
\begin{eqnarray*}
V_{1}= \mathbb{R}^{3} ,\quad V_{2} = \{0\}\times\mathbb{R}^{2}, \quad
V_{3} = {span}\lbrace (1,0,0)\rbrace.
\end{eqnarray*}
Then $V = \lbrace V_{i}\rbrace_{i=1}^{3}$ is an alternate dual of Parseval fusion frame $W= \lbrace W_{i}\rbrace_{i=1}^{3}$, in which
\begin{eqnarray*}
W_{1} = \textit{span}\{(0,0,1)\}, \quad W_{2} = \textit{span}\{(0,1,0)\}, \quad W_{3} = \textit{span}\{(1,0,0)\}.
\end{eqnarray*}
On the other hand, letting
\begin{eqnarray*}
U_{1} = W_{1}, \quad U_{2} = W_{2}, \quad U_{3} = {\textit{span}}\{(1,\frac{1}{50},0)\}.
\end{eqnarray*}
Then $ \lbrace U_{i}\rbrace_{i\in I}$ is a $\epsilon$-perturbation of $W$ with $\epsilon < 0.02$. Using the fact that
\begin{eqnarray*}
0.02 < \dfrac{1 - \sqrt{2} \Vert I_{\mathcal{H}} - S_{U}^{-1}\Vert} {\sqrt{2} \Vert S_{U}^{-1}\Vert}.
\end{eqnarray*}
we obtain $V$ is an approximate alternate dual fusion frame of $\lbrace U_{i}\rbrace_{i\in I}$ by Theorem \ref{perturbation2}.
\end{ex}

We know that many concepts of the classical frame theory have not been generalized to the fusion frames. For example, in the duality discussion, if $V = \lbrace (V_{i}, \upsilon_{i})\rbrace_{i\in I}$ is an alternate dual of fusion frame $W = \lbrace (W_{i}, \omega_{i})\rbrace_{i\in I}$, then
 $W$ is not an alternate dual fusion frame of $V$. Indeed, take
\begin{eqnarray*}
W_{1} = \textit{span}\lbrace (1,0,0)\rbrace, \quad W_{2} = \textit{span}\lbrace (1,1,0)\rbrace,\\ \quad W_{3} = \textit{span}\lbrace (0,1,0)\rbrace, \quad
W_{4} = \textit{span}\lbrace (0,0,1)\rbrace,
\end{eqnarray*}
and $\omega_{1} = \omega_{3} = \omega_{4} = 1$, $\omega_{2} =  \sqrt{2}$. Then $W = \lbrace (W_{i}, \omega_{i})\rbrace_{i\in I}$ is a fusion frame for $\mathbb{R}^{3}$ with an alternate dual as $V = \lbrace (V_{i}, \upsilon_{i})\rbrace_{i\in I}$ where
\begin{eqnarray*}
V_{1} = \textit{span}\lbrace (0,1,0)\rbrace, \quad V_{2} = \mathbb{R}^{3},\quad V_{3} = \textit{span}\lbrace (1,0,0)\rbrace,\quad V_{4} = \textit{span}\lbrace (0,0,1)\rbrace,
\end{eqnarray*}
and $\upsilon_{1} = \upsilon_{3} = 3$ , $\upsilon_{2} = 3\sqrt{2}$, $\upsilon_{4} = 1$, see Example 3.1 of \cite{Amiri}. A straightforward calculation shows that  $W$ is not an alternate dual fusion frame of $V$. Moreover,  for an alternate dual fusion frame $V$ of $W$, the fusion frame  $W$ is not an approximate alternate dual fusion frame of $V$, in general. The next theorem gives a sufficient condition for a fusion frame is approximate alternate  dual of its dual.

\begin{thm}\label{dual . app}
 Let $\lbrace (V_{i}, \upsilon_{i})\rbrace_{i\in I}$ be an alternate  dual of fusion frame $\lbrace (W_{i}, \omega_{i})\rbrace_{i\in I}$ such that
 \begin{eqnarray*}\label{dual.app}
\Vert  S_{W}^{-1} - S_{V}^{-1}\Vert < \Vert S_{W}\Vert^{-1/2} \Vert S_{V}\Vert^{-1/2}.
\end{eqnarray*}
Then $\lbrace (W_{i}, \omega_{i})\rbrace_{i\in I}$ is an approximate alternate dual fusion frame of $\lbrace (V_{i}, \upsilon_{i})\rbrace_{i\in I}$.
\end{thm}
\begin{proof}
By the assumption $T_{V} \phi_{vw} T_{W}^{*} = I_{\mathcal{H}}$, where $\phi_{vw}$ is given by (\ref{phi}). Also, it is not difficult to see that $\phi^{*}_{vw} \lbrace f_{i}\rbrace = \lbrace \pi_{W_{i}} S_{W}^{-1} f_{i}\rbrace$ for all $\lbrace f_{i}\rbrace\in \sum_{i\in I}\oplus V_{i}$. Hence
\begin{eqnarray*}
\Vert I_{\mathcal{H}} - T_{W} \phi_{wv} T_{V}^{*}\Vert &=& \Vert T_{W} \phi_{wv} T_{V}^{*} - T_{W} \phi_{vw}^{*} T_{V}^{*}\Vert\\
&\leq& \Vert T_{W}\Vert \Vert T_{V}\Vert \Vert \phi_{wv} - \phi_{vw}^{*}\Vert\\
&\leq& \Vert T_{W}\Vert \Vert T_{V}\Vert \Vert S_{W}^{-1} - S_{V}^{-1}\Vert < 1.
\end{eqnarray*}
\end{proof}
The fusion frame $W$ in Example \ref{example} is not an alternate  dual of $V$, however, a straightforward calculation shows that
\begin{eqnarray*}
\Vert S_{V}^{-1} - S_{W}^{-1}\Vert = \dfrac{1}{2},\quad  \Vert S_{V} \Vert = 2.
\end{eqnarray*}
 Hence, $W$ is an approximate alternate  dual of $V$ by Theorem \ref{dual . app}.
It is worth noticing that, unlike discrete frames, $\lbrace \psi_{wv}^{-1}W_{i}\rbrace_{i=1}^{3}$ is not dual of $\lbrace V_{i}\rbrace_{i=1}^{3}$. Indeed $\psi_{wv}^{-1} = 2 I_{\mathcal{H}}$ and so
\begin{eqnarray*}
\sum_{i\in I} \pi_{\psi_{wv}^{-1}W_{i}} S_{V}^{-1} \pi_{V_{i}} = \dfrac{1}{2} I_{\mathcal{H}}.
\end{eqnarray*}


\bibliographystyle{amsplain}

\end{document}